\documentclass[12pt]{amsart}

\usepackage{amscd,mathrsfs,epsfig,amsthm,amssymb,amsmath}
\usepackage[all]{xy} 
\usepackage[T1]{CJKutf8}
\usepackage[sans]{dsfont}  
\usepackage{hyperref}
\usepackage{tabularx}

\newtheorem{theorem}{Theorem}[section]
\newtheorem{corollary}[theorem]{Corollary}
\newtheorem{lemma}[theorem]{Lemma}
\newtheorem{definition}[theorem]{Definition}
\newtheorem{proposition}[theorem]{Proposition}
\newtheorem{remark}[theorem]{Remark}

\newtheorem{example}[theorem]{Example}

\newtheorem*{thma}{Theorem A}
\newtheorem*{thmb}{Theorem B}
\newtheorem*{thmc}{Theorem C}
\newtheorem*{thmd}{Theorem D}

\allowdisplaybreaks[2]

\def\calC{{\mathcal C}}
\def\calD{{\mathcal D}}
\def\calE{{\mathcal E}}

\def\calJ{{\mathcal J}}

\def\calO{{\mathcal O}}
\def\calP{{\mathcal P}}
\def\calQ{{\mathcal Q}}

\def\frakM{{\mathfrak M}}
\def\frakN{{\mathfrak N}}
\def\frakF{{\mathfrak F}}
\def\frakG{{\mathfrak G}}

\def\k{{\underline{k}}}

\def\ct'{{\rm C_2'}}

\def\k4{{\rm C_2 \times C_2}}
\def\c4{{{\rm C}_4}}
\def\d8{{{\rm D}_8}}

\def\Aut{\mathop{\rm Aut}\nolimits}

\def\coh{{\mathop{\rm coh}\nolimits}}

\def\Hom{\mathop{\rm Hom}\nolimits} 
\def\Nat{\mathop{\rm Nat}\nolimits}

\def\Id{\mathop{\rm Id}\nolimits}

\def\lim{\mathop{\varinjlim}\nolimits}
 
\def\Ob{\mathop{\rm Ob}\nolimits} 
\def\Mor{\mathop{\rm Mor}\nolimits}

\def\PSh{\mathop{\rm PSh}\nolimits}
\def\Sh{\mathop{\rm Sh}\nolimits}

\def\cod{\mathop{\rm cod}}
\def\dom{\mathop{\rm dom}}

\DeclareMathOperator{\OG}{\mathbf{O}\it{G}}

\DeclareMathOperator{\C}{\mathbf{C}}
\DeclareMathOperator{\D}{\mathbf{D}}
\DeclareMathOperator{\G}{\mathbf{G}}
\DeclareMathOperator{\bP}{\mathbf{P}}
\DeclareMathOperator{\Mod}{\mathfrak{M}od-}

\DeclareMathOperator{\rMod}{Mod-}
\DeclareMathOperator{\rmod}{mod-}
\DeclareMathOperator{\rSet}{Set-}

\begin{document}

\title{Skew category algebras and modules on ringed finite sites}

\author{Mawei Wu}
\author{Fei Xu}
\email{19mwwu@stu.edu.cn}
\email{fxu@stu.edu.cn}
\address{Department of Mathematics\\Shantou University\\Shantou, Guangdong 515063, China}

\subjclass[2020]{}


\thanks{The authors \begin{CJK*}{UTF8}{}
\CJKtilde \CJKfamily{gbsn}(吴马威、徐斐)
\end{CJK*} are supported by the NSFC grants No.12171297 and No.11671245}

\date{}
\maketitle

\begin{abstract}
Let $\calC$ be a small category. We investigate ringed sites $(\C,\mathfrak{R})$ on $\calC$ and the resulting module categories $\Mod\mathfrak{R}$. When $\calC$ is finite, based on Grothendieck and Verdier's classification of finite topoi, we prove that each $\Mod\mathfrak{R}$ is equivalent to $\rMod\mathfrak{R}|_{\calD}[\calD]$, where $\mathfrak{R}|_{\calD}[\calD]$ is the skew category algebra, canonically defined on $(\C,\mathfrak{R})$, for a uniquely determined full subcategory $\calD\subset\calC$ and the restriction $\mathfrak{R}|_{\calD}$ of $\mathfrak{R}$ to $\calD$.     
\end{abstract}

\tableofcontents

\section{Introduction}
Let $\calC$ be a small category. It is said to be \textit{finite} if Mor$\calC$ is finite, and it is \textit{object-finite} if $\Ob\calC$ is finite. One may put a Grothendieck topology $\calJ$ on $\calC$ to obtain a site $\C=(\calC,\calJ)$, which is said to be finite if the underlying category $\calC$ is. If $\mathfrak{R}$ is a sheaf of rings on $\C$, one can consider the \textit{right} modules on the ringed site $(\C,\mathfrak{R})$. They form an Abelian category, written as $\Mod\mathfrak{R}$, which is our major subject of investigation. In this paper, we mainly focus on finite categories, because all of their Grothendieck topologies, as well as the resulting topoi, are classified in \cite{SGA}. The classification enables us to characterize module categories on ringed finite sites.

Our motivation comes from the representation theory of finite groups and finite-dimensional algebras. This is the main reason that we want to focus on sites rather than topoi. On the one hand, in group representations, one often looks into various finite categories associated to the group in question, and consider their representations (i.e. presheaves) and cohomology. Recent progresses showed that sheaves on these categories are interesting \cite{Ba, XX}. On the other hand, representations of quivers, posets and general small categories have been under intensive investigations, in both algebra and topology, see for instance  \cite{AKO, ARS, L}. This work might shed new light on these subjects.

Suppose $k$ is a commutative ring with identity. We shall denote by $k$-Alg the category of unital associative algebras and unital $k$-algebra homomorphisms. Let $\mathfrak{R}$ be a presheaf of unital $k$-algebras on $\calC$. Motivated by the \textit{Grothendieck construction} $Gr_{\calC}\mathfrak{R}$ of $\mathfrak{R}$ on $\calC$ \cite[VI.8]{SGA1}, we introduce the \textit{skew category algebra} $\mathfrak{R}[\calC]$, as a ``linearization'' of $Gr_{\calC}\mathfrak{R}$, which includes skew group algebras and category algebras as special cases. The right modules of $\mathfrak{R}[\calC]$ form the module category ${\rm Mod}\text{-}\mathfrak{R}[\calC]$. When $\calC$ is a poset, the above construction has been studied by  Gerstenhaber and Schack \cite{GS88}, in disguise, as what they called a diagram ring $\mathfrak{R}!$.

We establish in Section 3 the following category equivalence.

\begin{thma} [Theorem \ref{skew}]
	Let $\calC$ be a small category and let $\mathfrak{R}:\calC^{op} \to k\mbox{\rm -Alg}$ be a presheaf of unital $k$-algebras on $\C$. If $\Ob \calC$ is finite, then we have a category equivalence 
	$$
	\Mod \mathfrak{R} \simeq {\rm Mod \text{-}}\mathfrak{R}[\calC].
	$$  
\end{thma}

One can continue to put a Grothendieck topology $\calJ$ on $\calC$. If it is the minimal topology $\calJ_{\min}$, then sheaves on $(\calC,\calJ_{\min})$ and presheaves on $\calC$ are the same. When we consider another topology $\calJ$ on $\calC$, we have to employ theories of sheaves and topoi. Under the circumstance, the sheaf category (a Grothendieck topos) $\Sh(\calC,\calJ)$ becomes a subtopos of the presheaf topos $\PSh(\calC)=\Sh(\calC,\calJ_{\min})$. Grothendieck and Verdier \cite[IV.9]{SGA} pointed out that $\calJ=\calJ^{\calD}$ is uniquely determined by a strictly full Karoubian subcategory $\calD$ of $\calC$. Moreover $\Sh(\calC,\calJ^{\calD})\simeq\PSh(\calD)$. Here being Karoubian means that the category is idempotent-complete \cite{KS} or Cauchy-complete \cite{Jo}. It allows us to deduce the following result, when $\calC$ is finite.

\begin{thmb}[Theorem \ref{shc}] 
Let $\mathbf{C}=(\calC,\mathcal{J})$ be a finite site and $\mathfrak{R}:\calC^{op} \to k\mbox{\rm -Alg}$ be a sheaf of unital $k$-algebras on $\C$. Then $\calJ=\calJ^{\calD}$ is determined by a strictly full subcategory $\calD$, and there is a category equivalence 
	$$
	\Mod\mathfrak{R}\simeq\rMod\mathfrak{R}|_{\calD}[\calD],
	$$
where $\mathfrak{R}|_{\calD}$ is the restriction of $\mathfrak{R}$ to $\calD$.  
\end{thmb}

If, for each $x \in \Ob \calD$, $\mathfrak{R}(x)$ is Noetherian, then $\mathfrak{R}|_{\calD}[\calD]$ is finite-dimensional and the coherent $\mathfrak{R}|_{\calD}[\calD]$-modules are the finitely generated ones. 

\begin{thmc}[Corollary \ref{cmshc}]
Let $\mathbf{C}=(\calC,\mathcal{J}^{\calD})$ be a finite site and $\mathfrak{R}:\calC^{op} \to k\mbox{\rm -Alg}$ be a sheaf of $k$-algebras on $\C$. If for each $x \in \Ob \calD$, $\mathfrak{R}(x)$ is Noetherian, then there is a category equivalence 	
	$$
	\coh \text{-} \mathfrak{R} \simeq \rmod\mathfrak{R}|_{\calD}[\calD].
	$$
\end{thmc}

The paper is organized as follows. In Section 2, we recall basics on Grothendieck topologies, topoi, ringed sites and their modules, and present some concrete examples. Then we characterize the modules on object-finite categories in Section 3. We recall Grothendieck and Verdier's classification of finite topoi in Section 4, which allows us to reduce the characterizations of general module categories to the situation in the preceding section. Afterwards, we can present the category equivalences that we have stated above. We end in Section 5 with several remarks.

\section{Prelimenaries}

In this section, we will recall the definitions of a Grothendieck topology, a sheaf, as well as the concepts of a ringed site, a ringed topos and their modules.\\
 
\noindent\textbf{Convention}: We shall use curly upper-case letters $\calC, \calD, \calE$ etc for categories (underlying sites); lower-case letters $w,x,y,z$ etc for the objects of categories; scan Serif letters $\textsf{f}, \textsf{g}, \textsf{h}, \textsf{u}, \textsf{v}$ etc for the morphisms of categories. Functors among these categories will be denoted by Greek letters. Presheaves and sheaves are written as Fraktur letters $\mathfrak{F}, \mathfrak{G},\mathfrak{M}, \mathfrak{N},\mathfrak{R}$ etc.

If $\textsf{f}$ is a morphism in $\calC$, we denote by ${\rm dom}(\textsf{f})$ and ${\rm cod}(\textsf{f})$ its domain and codomain.

In this paper, we shall let $k$ to be a commutative ring with identity, and denote by $k$-Alg the category of unital associative algebras and unital $k$-algebra homomorphisms. All algebras and algebra morphisms will be unital. All modules, either over an algebra or a ringed site, are \emph{right} modules. 

\subsection{Grothendieck topologies and sheaves}
Our main references on topos and sheaf theories are \cite{SGA, Stack, KS, MM}. Here we only recall necessary ingredients.

Let $\calC$ be a small category. We shall denote by $\PSh(\calC)$ the category of
presheaves of sets, or equivalently the contra-variant functors from $\calC$ to Set (the category of sets). Given the coefficient ring $k$, we write $\PSh(\calC,k)$ for the category of presheaves of $k$-modules (contra-variant functors from $\calC$ to Mod-$k$). If $k = \mathbb{Z}$, this is often dubbed as $\text{PSh}(\calC,\mathbb{Z}) = \text{PSh}(\calC, {\rm Ab})$. If $G$ is a group, regarded as a category with a single object, $\text{PSh}(G)$ and $\text{PSh}(G,k)$ are canonically identified  $\rSet G$ and Mod-$kG$, the categories of right $G$-sets and of right $kG$-modules, respectively. 

\subsubsection{Grothendieck topologies} One may put a Grothendieck topology $\mathcal{J}_{\calC}$ on $\calC$, via the concept of a sieve. By definition a \emph{sieve} $S$ on $x \in \text{Ob}\calC$ is simply a subfunctor of the representable functor $\text{Hom}_{\calC}(-,x)$. It can be identified with a set (still written as $S$) of morphisms with codomain $x$, satisfying the condition that if $\textsf{u} \in S$ and \textsf{uv} exists then $\textsf{uv} \in S$. For instance the \emph{maximal sieve} $\text{Hom}_{\calC}(-,x)$ is given by the set of all morphisms with codomain $x$. Meanwhile there is always an empty sieve on $x$.

\begin{definition} \label{grotop} A \emph{(Grothendick) topology} on a category $\calC$ is a function $\mathcal{J}$ which assigns to each object $x$ of $\calC$ a non-empty collection $\mathcal{J}(x)$ of sieves on $\calC$, in such a way that
	\begin{enumerate}
		\item the maximal sieve $\text{\rm Hom}_{\calC}(-,x)$ is in $\mathcal{J}(x)$;
		
		\item (stability axiom) if $S \in \mathcal{J}(x)$, then $\textsf{f}^*(S)=\{\textsf{g} ~|~ \textsf{fg} \in S\} \in \mathcal{J}(y)$ for any arrow $\textsf{f}: y \rightarrow x$;
		
		\item (transitivity axiom) if $S_1 \in \mathcal{J}(x)$, and $S_2$ is any sieve on $x$ such that $\textsf{f}^*(S_2) \in \mathcal{J}(y)$ for all $\textsf{f}: y \rightarrow x$ in $S_1$, then $S_2 \in \mathcal{J}(x)$.
	\end{enumerate}
\end{definition}

Any sieve in $\mathcal{J}(x)$ is called a \emph{covering sieve} on $x$.  A small category $\calC$
equipped with a Grothendieck topology $\mathcal{J}$ is called a \emph{site} $\textbf{C} = (\calC, \mathcal{J})$.

Here are some examples of Grothendieck topologies that we will encounter.

\begin{example}
	\begin{enumerate}
		\item (\emph{The minimal topology}) Given any category $\calC$, one can define the minimal topology $\mathcal{J}_{\rm min}$ such that $\mathcal{J}_{\rm min}(x) = \{ \text{\rm Hom}_{\calC}(-,x) \}$, for each $x \in \Ob\calC$. 
		
		\item (\emph{The dense topology})
		The dense topology can be defined for an arbitrary category $\calC$: for a sieve $S$, let $S \in \mathcal{J}_{\rm den}(x)$ if and only if for any $\textsf{f}: y \rightarrow x$ there is a morphism $\textsf{g}: z \rightarrow y$ such that $\textsf{fg} \in S$.
		
		It can be seen from the definition that, the covering sieves of a dense topology are always non-empty.
		
		\item (\emph{The maximal topology})
		Let $\calC$ be any category. To define the maximal topology $\calJ_{\rm max}$, we declare all sieves to be covering sieves.
	\end{enumerate}
\end{example}

\subsubsection{Sheaves on sites} Given a Grothendieck topology, one can consider sheaves on $\calC$. Roughly speaking, a sheaf of sets on $\calC$ is a presheaf of sets on $\calC$, satisfying
certain gluing properties mandated by $\mathcal{J}$. A concise definition
of a sheaf is the following.

\begin{definition} \label{sheaf}
A presheaf $\mathfrak{F} \in \PSh(\calC)$ is a ($\mathcal{J}$-)\emph{sheaf} of sets if, for every
$x \in \Ob \calC$ and every $S \in \calJ(x)$, the inclusion $S \hookrightarrow \Hom_{\calC}(-, x)$ induces an isomorphism
$$
\Nat(\Hom_{\calC}(-, x), \mathfrak{F}) \overset{\cong} \longrightarrow \Nat(S, \mathfrak{F}).
$$
The category of sheaves of sets on $\mathbf{C} = (\calC, \mathcal{J})$ is the full subcategory of $\PSh(\calC)$, consisting of all ($\mathcal{J}$-)sheaves, denoted by $\Sh(\mathbf{C})$.
\end{definition}

Let $\frakF_{\bullet}$ be the presheaf that sends every object $x$ to a fixed singleton set $\bullet$, and every morphism to the identity. It is a terminal object of $\PSh(\calC)$, and one quickly verifies that it is always a sheaf, under any topology.

Given a presheaf, there is a natural way to build a sheaf. By standard procedures of construction (see for instance \cite[III.5]{MM} or \cite[7.10]{Stack}), given a presheaf $\mathfrak{F}$, we shall first construct the half-sheafification (terminology by P. Johnstone) $\mathfrak{F}^\dag \in \text{PSh}(\calC)$ which, for each $x \in \text{Ob}\calC$, we define 
$$
\mathfrak{F}^\dag(x)=\lim_{\calJ(x)} \text{Nat}(-, \mathfrak{F}),
$$ 
where $\calJ(x)$ is deemed as a poset by inclusion. If $\mathfrak{F}^\dag$ is not a sheaf, then we repeat the construction and $\mathfrak{F}^a = (\mathfrak{F}^\dag)^\dag \in \text{Sh}(\mathbf{C})$ is the \emph{sheafification} of $\mathfrak{F}$ or the sheaf associated to $\mathfrak{F}$. The sheafification functor is exact and is left adjoint to the forgetful functor $\text{Sh}(\mathbf{C}) \rightarrow \text{PSh}(\calC)$, which is left exact.

\begin{remark}
From the adjunction between the forgetful functor and the sheafification functor, one realizes that $\Sh(\mathbf{C})$ is a localization of $\PSh(\calC)$ \cite{SGA, Jo, KS}. In fact the morphisms that become isomorphisms after sheafification form a multiplicative system $\Xi$ in $\PSh(\calC)$. The morphisms in $\Xi$ are called \textit{local isomorphisms} \cite{KS} or \textit{morphismes bicouvrant} \cite{SGA}, giving a category equivalence
$$
\Sh(\C)\simeq\Xi^{-1}\PSh(\calC).
$$
\end{remark}

\begin{example} 
	\begin{enumerate}
		\item If $\calC$ is given the minimal topology, then all presheaves are sheaves. The local isomorphisms are exactly all the natural isomorphisms. 
		
		\item If $\calC$ is given the maximal topology, then the only sheaf is $\frakF_{\bullet}$. The local isomorphisms are all the morphisms. (Be aware of set-theoretic issues!) 
	\end{enumerate}
\end{example}

Let $k$ be the coefficient ring. A presheaf of $k$-modules is a contra-covariant functor $\frakM$ from $\calC$ to $\rMod k$. It is a sheaf of $k$-modules if the composite $\calC{\buildrel{\frakM}\over{\to}}\rMod k \hookrightarrow$Set is a sheaf of sets. The resulting sheaf category $\Sh(\C, k)$ is Abelian, while the zero object $0$(=$\frakF_{\bullet}$) maps every object to $0$. The forgetful and sheafification functors restrict to a pair of adjoint functors between $\Sh(\C, k)$ and $\PSh(\calC, k)$. It follows that
$$
\Sh(\C, k)\simeq\PSh(\calC, k)/\mathcal{S},
$$
where $\mathcal{S}$ is the Serre subcategory consisting of $\frakM\in\PSh(\calC, k)$ satisfying $\frakM^a=0$.

In the Abelian category $\Sh(\C, k)$, one may consider its coherent objects. They form an Abelian full subcategory $\coh(\C, k)$, see \cite{Po}, as well as \cite{SGA, Stack}.

\begin{example} 
On a group $G$, considered as a category with one object $\bullet$, there are only two topologies. The minimal topology(=dense topology) contains only one sieve $\text{\rm Hom}_{\calC}(-,\bullet)$, while the maximal topology has two sieves $\{\emptyset, \text{\rm Hom}_{\calC}(-,\bullet)\}$. 
	
Under the minimal topology, we have $\Sh(\G)=\rSet G$ and $\Sh(\G, k)=\rMod kG$. When $k$ is Noetherian, $\coh(\G, k)=\rmod kG$, the category of finitely generated $kG$-modules.
\end{example}

\subsection{Geometric morphisms and the Comparison Lemma}

To compare sheaves on two sites, we need the concepts of continuous functors \cite{SGA}. Let $\alpha : \calC \to \calD$ be a (covariant) functor. Then it induces a restriction along $\alpha$, $Res_{\alpha} : \PSh(\calD) \to \PSh(\calC)$, which has two adjoint functors $LK_{\alpha}, RK_{\alpha} :\PSh(\calC) \to \PSh(\calD)$, called the left and right Kan extensions along $\alpha$.

\begin{definition} Let $\C=(\calC,\calJ)$ and $\D=(\calD,\calJ')$ be two sites. A functor $\alpha: \calC\to\calD$ is called \emph{continuous} if the restriction  $Res_{\alpha} : \PSh(\calD) \to \PSh(\calC)$ preserves sheaves.
\end{definition}

When $\alpha$ is continuous, it gives rise to a functor $Res_{\alpha} : \Sh(\D) \to \Sh(\C)$.

\begin{definition} Let $\C=(\calC,\calJ)$ and $\D=(\calD,\calJ')$ be two sites. A \emph{morphism of sites} $\Lambda : \D\to\C$ is given by a continuous functor $\alpha : \calC\to\calD$ such that 
$$
LK_{\alpha}^a=(-)^a\circ LK_{\alpha} : \Sh(\C)\to\Sh(\D)
$$ is left exact.
\end{definition}

In topos theory, the notion of a topos is more important than that of a site. Although we would like to focus on sites, certain relevant key results are formulated in terms of topoi. 

\begin{definition} A \emph{(Grothendieck) topos} is a category that is equivalent to some $\Sh(\C)$, a category of sheaves of sets on a site $\C$. A \emph{(geometric) morphism of topoi} 
$$
\Psi=(\Psi^{-1},\Psi_*) : \Sh(\D)\to\Sh(\C)
$$ 
consists of a pair of functors $\Psi_* : \Sh(\D)\to\Sh(\C)$ and $\Psi^{-1}:\Sh(\C)\to\Sh(\D)$ such that $\Psi^{-1}$ is left exact and is left adjoint to $\Psi_*$.

If $\Psi_*$ happens to be an inclusion which is fully faithful, then $\Sh(\D)$ is said to be a \emph{subtopos} of $\Sh(\C)$.
\end{definition}

A morphism of sites $\Lambda : \D\to\C$ always induces a morphism of topoi $\Lambda=(\Lambda^{-1},\Lambda_*)$, where $\Lambda_*=Res_{\alpha} : \Sh(\D)\to\Sh(\C)$ and $\Lambda^{-1}=LK_{\alpha}^a : \Sh(\C)\to\Sh(\D)$. 

The Comparison Lemma will be used in Section 4, in order to establish our main results.

\begin{definition}
Let $(\calC, \calJ)$ be a site. A full subcategory $\calD$ of $\calC$ is said to be \emph{$\calJ$-dense} if for every object $x$ of $\calC$ the sieve generated by the family of arrows from objects in $\calD$ to $x$ is $\calJ$-covering.
\end{definition}

Given such a subcategory $\calD$, the Grothendieck topology $\calJ | _{\calD}$ on $\calD$ induced by $\calJ$ is defined as follows: for any sieve $S$ in $\calD$ on an object $x$, $S \in \calJ |_{\calD}(x)$ if and only if $\bar{S} \in \calJ(x)$, where $\bar{S}$ is the sieve in $\calC$ generated by the arrows in $S$.

\begin{theorem}[Comparison Lemma] \label{CL}
Let $(\calC, \calJ)$ be a site and $\calD$ be a full subcategory of $\calC$ which is $\calJ$-dense. Then the restriction $Res_{\iota}$, along $\iota : \calD \hookrightarrow \calC$, induce a functor $Res_{\iota} : \Sh(\calC,\calJ)\to \Sh(\calD,\calJ |_{\calD})$, which is an equivalence. The quasi-inverse is the right Kan extension $RK_{\iota}$ along $\iota$.
\end{theorem}

It is illuminating to have a concrete yet more sophiscated example other than groups. 

\begin{example} Let $\calP$ be the category $x {\buildrel{\textsf{f}}\over{\to}} y {\buildrel{\textsf{g}}\over{\to}} z$, which is essentially a poset. Since it is finite and Karoubian, there are exactly eight Grothendieck topologies on $\calP$, corresponding to the eight subposets (full subcategories, including the empty one) of $\calP$ (See Propositions 4.1 and 4.4). Since the subposets are determined by the subsets of $\calP$, we shall write for instance $xy$ for the subposet $x \to y$, for brevity. In fact, the poset of subposets of $\calP$ is opposite isomorphic to that of topologies on $\calP$.
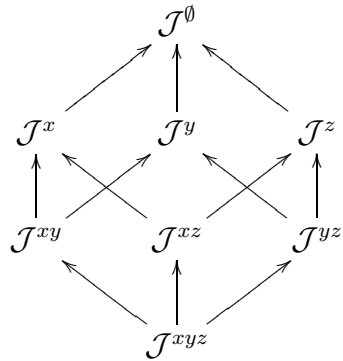
\[
\xymatrix{& \calJ^{\emptyset} & \\
\calJ^x \ar[ur] & \calJ^y \ar[u] & \calJ^z \ar[ul] \\
\calJ^{xy} \ar[u] \ar[ur] & \calJ^{xz} \ar[ur] \ar[ul] & \calJ^{yz} \ar[ul] \ar[u]\\
& \calJ^{xyz} \ar[ur] \ar[ul] \ar[u] &}
\]
Here $\calJ^{\emptyset}=\calJ_{\max}$, $\calJ^{xyz}=\calJ^{\calP}=\calJ_{\min}$, and $\calJ^{x}=\calJ_{\rm den}$. As to the rest, because there exists the minimal covering sieve on each object (for any given topology), we only need to list them since all sieves larger than the minimal ones must be covering. The minimal covering sieve of $\calJ^{\calQ}(w)$, for an object $w\in\Ob\calP$ and a subposet $\calQ$, is formed by all morphisms originated from an object of $\calQ$ and ending at $w$.

\begin{table}[htbp]
\caption{The MINIMAL covering sieve on each object}	
\centering
\newcolumntype{Z}{>{\centering\arraybackslash}X}	
\begin{tabularx}{\textwidth}{|c|Z|Z|Z|}
\hline
     & x & y & z  \\
\hline    
$\calJ^{xyz}$ & $\Hom_{\calP}(-,x)$ & $\Hom_{\calP}(-,y)$ & $\Hom_{\calP}(-,z)$  \\
\hline    
$\calJ^x$ & $\Hom_{\calP}(-,x)$ & $\{\textsf{f}\}$ & $\{\textsf{g}\textsf{f}\}$  \\
\hline		
$\calJ^y$ & $\emptyset$ & $\Hom_{\calP}(-,y)$ & $\{\textsf{g},\textsf{g}\textsf{f}\}$  \\
\hline
$\calJ^z$ & $\emptyset$ & $\emptyset$ & $\Hom_{\calP}(-,z)$ \\
\hline
$\calJ^{xy}$ & $\Hom_{\calP}(-,x)$ & $\Hom_{\calP}(-,y)$ & $\{\textsf{g},\textsf{g}\textsf{f}\}$ \\
\hline	
$\calJ^{xz}$ & $\Hom_{\calP}(-,x)$ & $\{\textsf{f}\}$ & $\Hom_{\calP}(-,z)$ \\
\hline
$\calJ^{yz}$ & $\emptyset$ & $\Hom_{\calP}(-,y)$ & $\Hom_{\calP}(-,z)$ \\
\hline	
$\calJ^{\emptyset}$ & $\emptyset$ & $\emptyset$ & $\emptyset$ \\
\hline	 		
\end{tabularx}
\end{table}

Let $\calQ$ be a subposet of $\calQ'$. Then $\calJ^{\calQ'}\subset\calJ^{\calQ}$ and the identity functor $\Id_{\calP}$ induces a morphism of sites $\bP_{\calQ'}=(\calP,\calJ^{\calQ'}) \to \bP_{\calQ}=(\calP,\calJ^{\calQ})$, and consequently a morphism of topoi $\Sh(\bP_{\calQ}) \to \Sh(\bP_{\calQ'})$ which is fully faithful and makes the former a subtopos of the latter. Indeed they form a lattice of subtopoi of $\PSh(\calP)$ as follows
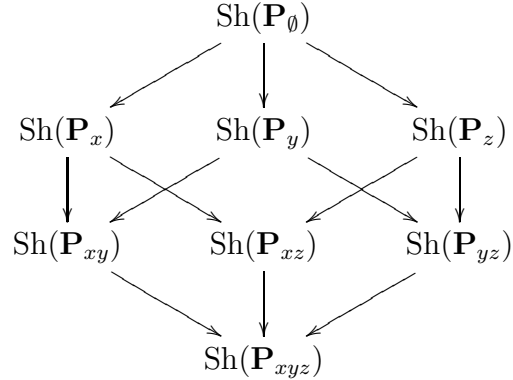
\[
\xymatrix{& \Sh(\bP_{\emptyset}) \ar[dr] \ar[d] \ar[dl] & \\
\Sh(\bP_x) \ar[dr] \ar[d] & \Sh(\bP_y) \ar[dr] \ar[dl] & \Sh(\bP_z) \ar[dl] \ar[d] \\
\Sh(\bP_{xy}) \ar[dr] & \Sh(\bP_{xz}) \ar[d] & \Sh(\bP_{yz}) \ar[dl]\\
& \Sh(\bP_{xyz}) &}
\]
By direct calculations via Definition 2.3, one readily deduces that $\Sh(\bP_{\calQ})\simeq\PSh(\calQ)$. For instance, $\Sh(\bP_{\emptyset})=\Sh(\bP_{\max})=\{\frakF_{\bullet}\}$, $\Sh(\bP_{xyz})=\Sh(\bP_{\calP})=\Sh(\bP_{\min})=\PSh(\calP)$, and $\Sh(\bP_{xy})$ consists of (pre)sheaves $\frakF$, satisfying $\frakF(y)=\frakF(z)$ and $\frakF(\textsf{g})=\Id$. 

As a subcategory of $\PSh(\calP)$, each sheaf category $\Sh(\bP_{\calQ})$ can be obtained by localization with respect to a  set of local isomorphisms, that is, $\Sh(\bP_{\calQ})\simeq\Xi^{-1}_{\calQ}\PSh(\calP)$. The set $\Xi_{\calQ}$ can be explicitly written out. For instance, $\Xi_{xy}$ consists of natural transformations that become isomorphisms after sheafification. It contains exactly the natural transformations of the following form
$$
\xymatrix{\frakF(x) \ar[d]_{\Id_{\frakF(x)}} & \frakF(y) \ar[l]_{\frakF(\textsf{f})} \ar[d]_{\Id_{\frakF(y)}} & \frakF(z) \ar[l]_{\frakF(\textsf{g})} \ar[d]^{\frakF(\textsf{g})}\\
\frakF(x) & \frakF(y) \ar[l]^{\frakF(\textsf{f})} & \frakF(y) \ar[l]^{\Id_{\frakF(y)}}}
$$
The upper row depicts a presheaf on $\calP$, and the lower one is its sheafification with respect to $\calJ^{xy}$, a sheaf on $\bP_{xy}$.

In the end, for the coefficient ring $k$, $\Sh(\bP_{\calQ},k)\simeq\PSh(\calQ,k)\simeq\rMod k\calQ$, where $k\calQ$ is the incidence algebra of $\calQ$. 
\end{example}

We hope the above example can be useful for the reader whose background is in the representation theory of finite-dimensional algebras, or of finite groups.

\subsection{Ringed sites, ringed topoi and their modules}

Let $\C$ be a site. We recall the definitions of a ringed site, a ringed topos, and their modules see \cite{SGA, Stack, KS}.

\begin{definition} Let $\C$ be a site.
\begin{enumerate}
\item A \emph{sheaf of $k$-algebras} on $\C$ is a presheaf $\mathfrak{R} : \calC^{op} \to k$-{\rm Alg} which becomes a sheaf of sets after composing with the canonical functor $k\text{-}{\rm Alg} \hookrightarrow {\rm Set}$.

\item A \emph{sheaf of $\mathfrak{R}$-modules} is a presheaf of right $\mathfrak{R}$-modules $\mathfrak{M}$, such that the underlying presheaf of $k$-modules $\mathfrak{M}$ is a sheaf of sets.

\item A \emph{morphism of sheaves of right $\mathfrak{R}$-modules} is a morphism of presheaves of $\mathfrak{R}$-modules.

\item The \emph{category of sheaves of $\mathfrak{R}$-modules} is denoted by $\Mod\mathfrak{R}$.
\end{enumerate}	
\end{definition}	

We will encounter both ringed sites and ringed topoi later on. From representation-theorists' perspective, we prefer working with ringed sites whenever it is possible, since they are comparable with group algebras etc.

\begin{definition}
\begin{enumerate}	
\item A \emph{ringed site} is a pair $(\C, \mathfrak{R})$ where $\C$ is a site and $\mathfrak{R}$ is a sheaf of rings on $\C$. The sheaf $\mathfrak{R}$  is called the \emph{structure sheaf} of the ringed site.

\item Let $(\C,\mathfrak{R})$ and $(\C',\mathfrak{R}')$ be two ringed sites. A \emph{morphism of ringed sites} $(\Lambda,\Lambda^{\sharp}) : (\C,\mathfrak{R})\to(\C',\mathfrak{R}')$ is given by a morphism of sites $\Lambda : \C\to\C'$ and a map of sheaves of $k$-algebras $\Lambda^\sharp : \Lambda^{-1}(\mathfrak{R}')\to\mathfrak{R}$, which by adjunction is the same thing as a map of sheaves of $k$-algebras $\Lambda^{\sharp}: \mathfrak{R}' \to \Lambda_*(\mathfrak{R})$.

\item A \emph{ringed topos} is a pair $(\Sh(\C), \mathfrak{R})$ where $\C$ is a site and $\mathfrak{R}$ is a sheaf of rings on $\C$. The sheaf $\mathfrak{R}$ is the \emph{structure sheaf} of the ringed topos.

\item Let $(\Sh(\C), \mathfrak{R})$, $(\Sh(\C'), \mathfrak{R}')$ be ringed topoi. A \emph{morphism of ringed topoi} $(\Psi,\Psi^{\sharp}): (\Sh(\C), \mathfrak{R}) \to (\Sh(\C'), \mathfrak{R}')$ is given by a morphism of topoi $\Psi=(\Psi^{-1},\Psi_*): \Sh(\C) \to \Sh(\C')$ together with a map of sheaves of $k$-algebras $\Psi^{\sharp}: \Psi^{-1} (\mathfrak{R}') \to \mathfrak{R}$, which by adjunction is the same thing as a map of sheaves of $k$-algebras $\Psi^{\sharp}: \mathfrak{R}' \to \Psi_*(\mathfrak{R})$.
\end{enumerate}		
\end{definition}

A ringed topos is just a ringed site. The difference is on their morphisms. A morphism between ringed sites always induces one between corresponding ringed topoi. For the experienced reader, it is possible to present our contents in this paper with ringed sites only. However we choose to use both ringed sites and ringed topoi in order to be consistent with references such as \cite{SGA}.

The following example may give the reader some feeling of the sheaves of modules.

\begin{example} Let $\calC$ be the following category
$$
\xymatrix{x \ar@(ur,ul)[]_{\textsf{1}_x} \ar@(dl,dr)[]_{\textsf{h}} \ar@<0.5ex>[rr]^{\textsf{f}} \ar@<-0.5ex>[rr]_{\textsf{g}} & & y \ar@(ur,ul)[]_{\textsf{1}_y} ,}
$$
such that $\textsf{h}^2=\textsf{1}_x$ and $\textsf{f}\textsf{h}=\textsf{g}$. By Propositions 4.1 and 4.4, there are four topologies.

A presheaf of $k$-algebras $\mathfrak{R}$ on $\calC$ is represented by the following picture
$$
\xymatrix{\mathfrak{R}(x) \ar@(ur,ul)[]_{\mathfrak{R}(\textsf{1}_x)} \ar@(dl,dr)[]_{\mathfrak{R}(\textsf{h})} & & \mathfrak{R}(y) \ar@<0.5ex>[ll]^{\mathfrak{R}(\textsf{f})} \ar@<-0.5ex>[ll]_{\mathfrak{R}(\textsf{g})} \ar@(ur,ul)[]_{\mathfrak{R}(\textsf{1}_y)},}
$$
which consists of two algebras and five compatible algebra homomorphisms. One can analogously draw any right $\mathfrak{R}$-module like this.

If $\calC$ is given a topology and $\mathfrak{R}$ becomes a sheaf of $k$-algebras, the module category $\Mod\mathfrak{R}$ is a subcategory of $\PSh(\calC,k)\simeq\rMod k\calC$, where $k\calC$ is the category algebra \cite{Xu}.
\end{example}

\subsection{Special categories} \label{ei} A full subcategory $\calD$ of $\calC$ is said to be {\it strictly full} if an object $x$ belongs to $\calD$ then every object isomorphic to $x$ in $\calC$ must lie in $\calD$. 

A category is \emph{EI} if all of its endomorphisms are isomorphisms \cite{L}. Groups and partially ordered sets are EI-categories. Given an EI-category $\calC$, there is a pre-order defined on Ob$\calC$, that is $x \leq y$ if and only if $\text{Hom}_{\calC}(x,y) \neq \emptyset$. Let $[x]$ be the isomorphism class of an object $x \in$ Ob$\calC$. This pre-order induces a partial order on the set Iso$\calC$ of isomorphism classes of Ob$\calC$ (specified by $[x] \leq [y]$ if and only if $\text{Hom}_{\calC}(x,y) \neq \emptyset$). For an EI-category $\calC$ and an object $x \in$ Ob$\calC$, we define a full subcategory $\calC_{\leq x} \subset \calC$, consisting of all $y \in \text{Ob}\calC$ such that $\text{Hom}_{\calC}(y,x) \neq \emptyset$. Similarly we can define several other full subcategories of $\calC$: $\calC_{<x}$, $\calC_{\geq x}$ and $\calC_{> x}$. 

\begin{definition} Let $\calC$ be an EI-category.
\begin{enumerate} 
\item We denote by $\calC_{\min}$ the full subcategory consisting of all minimal objects.

\item A (strictly) full subcategory $\calD$ is called a \emph{co-ideal}, if $x\in\Ob\calD$ implies that $\calC_{\le x}\subset\calD$.
\end{enumerate}
\end{definition}

Here is an EI-category related to a finite group.

\begin{example} Let $G$ be a finite group. The \emph{orbit category} $\calO(G)$ has objects the orbits $G/H$, where $H$ is a subgroup, with morphism sets
	$$
	{\rm Hom}_{\calO(G)}(G/H,G/K) =\{c_g:G/H\to G/K\bigm{|} g\in G\}.
	$$
Here $c_g$ is the $G$-map given by $H \mapsto gK$ (different group elements may induce the same map).
\end{example}

In practice, one often takes a certain collection of subgroups and then creates smaller orbit categories. For instance, one can focus on the $p$-orbit category $\calO_p(G)$ on $p$-subgroups, and the smaller orbit category $\calO^{\circ}_p(G)$ on non-identity $p$-subgroups of $G$. More examples of EI-categories, including the fusion systems, can be found in \cite{AKO}. To understand the representations (that is, functors) and cohomology of such categories (and their applications to group representation theory) is our motivation to examine sheaves over finite categories.

\begin{definition} A category is \emph{Karoubian} \cite[IV, Exercise 7.5]{SGA} (i.e. \emph{idempotent-complete} \cite[Exercise 2.9]{KS} or \emph{Cauchy-complete} \cite{Jo}), if every idempotent endomorphism in $\calC$ has an image.
\end{definition}

If $\calC$ is an EI-category, then the only idempotent endomorphisms are the identities. Therefore EI-categories are always Karoubian. 

\section{Modules over presheaves of algebras} \label{triv}

Under the minimal topology, all presheaves on $\calC$ are sheaves. This section only deals with presheaves of algebras and their modules. We will see that it is interesting in its own right, and moreover it is necessary for upcoming developments.

\subsection{Grothendieck construction} \label{grocon} 

Let $\calC$ be a small category and $\mathfrak{R}:\calC^{op} \to k\text{-Alg}$ be a presheaf of $k$-algebras. Since every unital ring can be regarded as a small additive category with a single object \cite{Mit}, $\mathfrak{R}$ becomes a contravariant functor from $\calC$ to ${\rm Cat}$, the category of small categories. 

There is a device, called the Grothendieck construction (see \cite[VI.8]{SGA1}), to build a fibred category over a small category $\calC$ equipped with a pseudo-functor to ${\rm Cat}$. Since strict functors are pseudo, we can examine the construction for $\mathfrak{R}$. 

\begin{definition}
	Let $\mathfrak{F}:\calC^{op} \to {\rm Cat}$ be a (strict) functor. The \emph{Grothendieck construction} $Gr_{\calC}\mathfrak{F}$ is a small category whose objects are $\{(x,c) ~|~ x \in \Ob \calC, c \in \Ob \mathfrak{F}(x)\}$, and whose morphisms are 
		$$
		{\rm Hom}_{Gr_{\calC}\mathfrak{F}}((x,c),(y, d))=\{(\textsf{f},\textsf{u}) ~|~ \textsf{f}: x \to y \in {\rm Mor}\calC, \textsf{u}: c \to \mathfrak{F}(\textsf{f})(d) \in {\rm Mor}~\mathfrak{F}(x)\}.
		$$		
The composition of morphisms is given by $(\textsf{g}, \textsf{v}) \circ (\textsf{f},\textsf{u}) = (\textsf{g}\textsf{f}, \mathfrak{F}(\textsf{f})(\textsf{v}) \circ \textsf{u})$.
\end{definition}

Since $\mathfrak{R}$ admits an ``action'' by $\calC$, our category $Gr_{\calC}\mathfrak{R}$ carries certain algebraic structures on its morphism sets. To explain, we present $Gr_{\calC}\mathfrak{R}$ as follows
\begin{enumerate}
	\item Objects: $\{(x, \bullet_x) ~|~ x \in \Ob\calC \; \mbox{and}\, \bullet_x \, \mbox{is the unique object of} \; \mathfrak{R}(x)\}$,
	
	\item Morphisms: 
	$${\rm Hom}_{Gr_{\calC}\mathfrak{R}}((x,\bullet_x),(y, \bullet_y))=\{(\textsf{f},r) ~|~ \textsf{f}: x \to y, r: \bullet_x \to \mathfrak{R}(\textsf{f})(\bullet_y), r \in \mathfrak{R}(x)\},$$
	
	\item Composition of morphisms: $(\textsf{g},s) \circ (\textsf{f},r) = (\textsf{g}\textsf{f},\mathfrak{R}(\textsf{f})(s)  r)$.
\end{enumerate}

Note that we always have $\mathfrak{R}(\textsf{f})(\bullet_y)=\bullet_x$ for any morphism $\textsf{f}: x \to y$.

\begin{proposition} Let $Gr_{\calC}\mathfrak{R}$ be the above Grothendieck construction. Suppose $\textsf{f} :x \to y$ is a morphism in $\calC$. Then
\begin{enumerate}
\item The following subset, of ${\rm Hom}_{Gr_{\calC}\mathfrak{R}}((x,\bullet_x),(y, \bullet_y))$,
$$
\Hom^\textsf{f}((x,\bullet_x),(y, \bullet_y))=\{(\textsf{f},r) ~|~ r: \bullet_x \to \mathfrak{R}(\textsf{f})(\bullet_y), r \in \mathfrak{R}(x)\}
$$
has a natural $k$-module structure.

\item The set $\Aut^{\textsf{1}_x}(x,\bullet_x)=\{(\textsf{1}_x,r) ~|~ r: \bullet_x \to \mathfrak{R}(\textsf{1}_x)(\bullet_x), r \in \mathfrak{R}(x)\}$ has a natural ring structure and is isomorphic to $\mathfrak{R}(x)$.

\item The set $\Hom^\textsf{f}((x,\bullet_x),(y, \bullet_y))$ is a rank one free right $\Aut^{\textsf{1}_x}(x,\bullet_x)$-module, with $(\textsf{f},1_{\mathfrak{R}(x)})$ as a base element.
\end{enumerate}
\end{proposition}

\begin{proof} On $\Hom^\textsf{f}((x,\bullet_x),(y, \bullet_y))$, one can define an addition by asking $(\textsf{f},r)+(\textsf{f},r')$ to be $(\textsf{f},r+r')$. For $l\in k$, one sets $l(\textsf{f},r)=(\textsf{f},lr)$ and then the operations do give a $k$-module structure. 

Specializing to $\Aut^{\textsf{1}_x}(x,\bullet_x)=\Hom^{\textsf{1}_x}((x,\bullet_x),(x, \bullet_x))$, it has an multiplicative structure on top of the $k$-module structure. In fact, we can set $(1_x,r)(1_x,r')=(1_x,rr')$ and verify that it is compatible with the $k$-module structure. It is easy to see that there exists a natural isomorphism $\Aut^{\textsf{1}_x}(x,\bullet_x)\cong\mathfrak{R}(x)$.

Finally, from $(\textsf{f},1_{\mathfrak{R}(x)})(1_x,r)=(\textsf{f},r)$, we can readily establish the third statement.
\end{proof}

Since there is no reasonable way to define a ``sum'' of $(\textsf{f},r)$ and $(\textsf{f}',r')$, the set ${\rm Hom}_{Gr_{\calC}\mathfrak{R}}((x,\bullet_x),(y, \bullet_y))$ is not an Abelian group itself. Nonetheless, motivated by the above analyse and the work of Mitchell \cite{Mit}, we introduce an algebra on the Grothendieck construction.

\subsection{Skew category algebras} \label{skewcat}
The idea is to consider the coproduct of all $k$-modules constructed in Proposition 3.2 (1). However for future applications, we shall slightly simplify the notations. 

\begin{definition}
	Let $\calC$ be a (non-empty) small category. Let $\mathfrak{R}:\calC^{op} \to k\text{-}{\rm Alg}$ be a presheaf of $k$-algebras. The \emph{skew category algebra} $\mathfrak{R}[\calC]$ on $\calC$ with respect to $\mathfrak{R}$ is a $k$-module spanned over elements of the form $r \textsf{f}$, where $\textsf{f} \in \Mor \calC$ and $r \in \mathfrak{R}(\dom(\textsf{f}))$. We define the multiplication on two base elements by the rule
	\begin{eqnarray}
		s\textsf{g} \ast r\textsf{f}=
		\begin{cases}
			(\mathfrak{R}(\textsf{f})(s)r) \textsf{g}\textsf{f},       & \text{if} ~{\rm dom}(\textsf{g})={\rm cod}(\textsf{f}); \notag \\
			0, & {\rm otherwise}.
		\end{cases}
	\end{eqnarray} 
Extending this product linearly to two arbitrary elements, $\mathfrak{R}[\calC]$ becomes an associative $k$-algebra.

The algebra $\mathfrak{R}[\calC]$ has an identity $\Sigma_{x\in\Ob\calC} 1_{\mathfrak{R}(x)}\textsf{1}_x$ if $\calC$ is object-finite.
\end{definition}

If $\calC$ is empty, the skew category algebra is understood to be the null ring. The above construction will be used to characterize the category of right $\mathfrak{R}$-modules over a presheaf $\mathfrak{R}$ of $k$-algebras.

\begin{remark} 
\begin{enumerate}
\item If $\mathfrak{R}=\underline{k}:\calC^{op}  \to k\text{-}{\rm Alg}$ is the constant presheaf such that $\underline{k}(x)=k$ for all $x \in \Ob \calC$, then the skew category algebra $\underline{k}[\calC]$ on $\calC$ with respect to $\underline{k}$ is just the category algebra $k\calC$ \cite{Xu}.
		
\item If $\calC$ be a group $G$, then $\mathfrak{R}=\mathfrak{R}(\bullet)$ ($\bullet$ being the unique object of $G$ as a category) is simply a $k$-algebra with a designated $G$-action, and the skew category algebra is just a skew group algebra. 

\item If $\calC=\calP$ is a poset, we recover the construction $\mathfrak{R}!$ of Gerstenhaber and Schack \cite{GS88}. 
\end{enumerate}
\end{remark}

We note that due to our convention the above constructions are slightly different from their common forms in the literature. For example, in (2) the group $G$ acts on $\mathfrak{R}$ on the right, while one usually sets $\mathfrak{R}$ to be a left $G$-module, see for instance \cite[III.4]{ARS}. If we consider the Grothendieck construction for covariant functors and the resulting skew category algebra, then the difference will disappear.

\subsection{Presheaves of modules}  

For a presheaf of $k$-algebras $\mathfrak{R}$ on an object-finite small category $\calC$, the following theorem characterizes the category $\Mod\mathfrak{R}$ of right $\mathfrak{R}$-modules.

\begin{theorem} \label{skew}
Let $\calC$ be a small category and $\mathfrak{R}:\calC^{op} \to k\mbox{-}{\rm Alg}$ be a presheaf of $k$-algebras on $\C$. If $\Ob \calC$ is finite, then we have the following equivalence 	
$$
\Mod\mathfrak{R} \simeq \rMod \mathfrak{R}[\calC].
$$  
\end{theorem}

\begin{proof}
	We will prove by constructing a pair of functors, and then show that they give rise to an equivalence between the categories. Recall that the category $\Mod\mathfrak{R}$ is just the category of preasheaves of $\mathfrak{R}$-modules. 
	
	Firstly, let us define a functor $\Theta: \Mod\mathfrak{R} \to {\rm Mod \text{-}}\mathfrak{R}[\calC]$ by
	$$
	\mathfrak{M} \mapsto \bigoplus_{y \in \calC} \mathfrak{M}(y).
	$$

	Here $\bigoplus_{y \in \calC} \mathfrak{M}(y)$ is a right $\mathfrak{R}[\calC]$-module:	
	for any $m \in \mathfrak{M}(y)$ and any $r\textsf{f} \in \mathfrak{R}[\calC]$, we can define an action 	
\begin{eqnarray}
m \cdot r\textsf{f} = 
\begin{cases}
\mathfrak{M}(\textsf{f})(m) \cdot r, & \text{if}~ \cod(\textsf{f})=y; \notag \\
0, & \text{otherwise.}\\
\end{cases}
\end{eqnarray}	

	We are now at the position to verify that the operation defined above does give $\frakM$ a right $\mathfrak{R}[\calC]$-module structure. It is clear that the operation is closed. For the identity $1_{\mathfrak{R}[\calC]} \in \mathfrak{R}[\calC]$, since
	$$
	m \cdot 	1_{\mathfrak{R}(y)}\textsf{1}_y=\mathfrak{M}(\textsf{1}_y)(m)\cdot 1_{\mathfrak{R}(y)}=m,
	$$
	it follows that	
	$$
	m \cdot 1_{\mathfrak{R}[\calC]}=m \cdot \sum_y 1_{\mathfrak{R}(y)}\textsf{1}_y=m.
	$$
	It remains to show that 
	$$
	(m \cdot s\textsf{g}) \cdot r\textsf{f} = m \cdot (s\textsf{g} \ast r\textsf{f})
	$$
	for any two elements $s\textsf{g}, r\textsf{f} \in \mathfrak{R}[\calC]$. We only need to consider the case that $x \stackrel{\textsf{f}} \to y \stackrel{\textsf{g}} \to z$ are composable in $\calC$, $m \in \mathfrak{M}(z)$, $s \in \mathfrak{R}(y)$ and $r \in \mathfrak{R}(x)$.
	Then we have	
	\begin{align*}
	& (m \cdot s\textsf{g}) \cdot r\textsf{f} \\
	= & (\mathfrak{M}(\textsf{g})(m)\cdot s) \cdot r\textsf{f} \\
	= & \mathfrak{M}(\textsf{f})(\mathfrak{M}(\textsf{g})(m) \cdot s) \cdot r, 
	\end{align*}
	and
	\begin{align*}
	& m \cdot (s\textsf{g} \ast r\textsf{f}) \\
	= & m \cdot (\mathfrak{R}(\textsf{f})(s) r) \textsf{g}\textsf{f} \\
	= &\mathfrak{M}(\textsf{g}\textsf{f})(m) \cdot (\mathfrak{R}(\textsf{f})(s) r)\\ 
	= & \mathfrak{M}(\textsf{f})(\mathfrak{M}(\textsf{g})(m)) \cdot (\mathfrak{R}(\textsf{f})(s)r) \\
	= & (\mathfrak{M}(\textsf{f})(\mathfrak{M}(\textsf{g})(m)) \cdot \mathfrak{R}(\textsf{f})(s)) \cdot r. \\
	\end{align*}
	Due to the right $\mathfrak{R}$-module structure on $\mathfrak{M}$, it follows that
	$$
	\mathfrak{M}(\textsf{f})(\mathfrak{M}(\textsf{g})(m) \cdot s) = \mathfrak{M}(\textsf{f})(\mathfrak{M}(\textsf{g})(m)) \cdot \mathfrak{R}(\textsf{f})(s).
	$$
	Hence
	$$
	(m \cdot s\textsf{g}) \cdot r\textsf{f} = m \cdot (s\textsf{g} \ast r\textsf{f}).
	$$
	
With obvious map on morphisms, it is not difficult to see that $\Theta$ is a functor. 

Secondly, let us define another functor 
$$
\Omega: {\rm Mod \text{-}}\mathfrak{R}[\calC] \to \Mod\mathfrak{R}
$$ by $N \mapsto \mathfrak{N}$. The presheaf $\mathfrak{N} : \calC^{op} \to \rMod k$ is defined as follows
	$$
	\frakN(y)= N \cdot 1_{\mathfrak{R}(y)}\textsf{1}_y, \forall y\in\Ob\calC
	$$
	and
	$$ 
	\frakN(\textsf{f})= - \cdot 1_{\mathfrak{R}(x)}\textsf{f} : N \cdot 1_{\mathfrak{R}(y)}\textsf{1}_y \to N \cdot 1_{\mathfrak{R}(x)}\textsf{1}_x, \forall \textsf{f} : x \to y. 
	$$
	Our task now is to verify that $\mathfrak{N}$ is a right $\mathfrak{R}$-module. Recall that, a presheaf of $\mathfrak{R}$-module $\mathfrak{N}$ is a presheaf $\mathfrak{N}: \calC^{op} \to \rMod k$ such that, for each $x \in \Ob \calC$, $\mathfrak{N}(x)$ has a right $\mathfrak{R}(x)$ action.

	Let us first check that $\mathfrak{N}$ is a functor. For any $n\cdot 1_{\mathfrak{R}(y)}\textsf{1}_y \in N \cdot 1_{\mathfrak{R}(y)}\textsf{1}_y$, observe that
	$$(n\cdot 1_{\mathfrak{R}(y)}\textsf{1}_y)\cdot 1_{\mathfrak{R}(x)}\textsf{f} = n \cdot (1_{\mathfrak{R}(x)} \textsf{f} \ast 1_{\mathfrak{R}(x)}\textsf{1}_x)= (n \cdot 1_{\mathfrak{R}(x)} \textsf{f}) \cdot 1_{\mathfrak{R}(x)}\textsf{1}_x \in N \cdot 1_{\mathfrak{R}(x)}\textsf{1}_x,$$ 
	so $\mathfrak{N}(\textsf{f})=- \cdot 1_{\mathfrak{R}(x)} \textsf{f}$ is well defined. Obviously,
	$$\mathfrak{N}(\textsf{1}_y)=- \cdot 1_{\mathfrak{R}(y)} \textsf{1}_y={\rm Id}_{	N \cdot 1_{\mathfrak{R}(y)}\textsf{1}_y},$$
	and
	\begin{align*}
	&\mathfrak{N}(\textsf{f}) \circ \mathfrak{N}(\textsf{g}) \\
	= & - \cdot 1_{\mathfrak{R}(y)} \textsf{g} \cdot 1_{\mathfrak{R}(x)} \textsf{f} \\
	= & - \cdot ( 1_{\mathfrak{R}(y)} \textsf{g} \ast 1_{\mathfrak{R}(x)} \textsf{f}) \\
	= & - \cdot (\mathfrak{R}(\textsf{f})(1_{\mathfrak{R}(y)}) 1_{\mathfrak{R}(x)})	\textsf{g}\textsf{f} \\ 
	= & - \cdot 1_{\mathfrak{R}(x)} \textsf{g}\textsf{f} \\
	= & \mathfrak{N}(\textsf{g}\textsf{f}). \\
	\end{align*}
		
	According to the fact that $\mathfrak{R}(\textsf{f})$ is a unital algebra homomorphism, the fourth equality above holds. Thus we have checked that $\mathfrak{N}$ is a functor.
	
	The next thing to do is to check that $\mathfrak{N}$ has a right $\mathfrak{R}$-module structure. For each $y \in {\rm Ob}\calC$, there is a right action of $\mathfrak{R}(y)$ on $N \cdot 1_{\mathfrak{R}(y)}\textsf{1}_y$ as follows. For any $n\cdot 1_{\mathfrak{R}(y)}\textsf{1}_y \in N \cdot 1_{\mathfrak{R}(y)}\textsf{1}_y$ and $r \in \mathfrak{R}(y)$, $r$ acts on $n\cdot 1_{\mathfrak{R}(y)}\textsf{1}_y$ by $r\textsf{1}_y$ from the right: $(n\cdot 1_{\mathfrak{R}(y)}\textsf{1}_y) \cdot r\textsf{1}_y=(n\cdot r\textsf{1}_y)\cdot 1_{\mathfrak{R}(y)}\textsf{1}_y \in N \cdot 1_{\mathfrak{R}(y)}\textsf{1}_y$. It is straightforward to verify that this makes $\mathfrak{N}$ a right $\mathfrak{R}$-module. 
	
From the definitions of $\Theta$ and $\Omega$, it is not difficult to see that
	$$
	\Omega\Theta \cong {\rm Id}_{\Mod\mathfrak{R}},
	$$ 
	and, moreover if ${\rm Ob}\calC$ is finite, 
	$$
	\Theta\Omega \cong {\rm Id}_{{\rm Mod \text{-}}\mathfrak{R}[\calC]}.
	$$
	
	Thus we arrive at the conclusion that if $\calC$ is object-finite then
	$$
	\Mod\mathfrak{R} \simeq {\rm Mod \text{-}}\mathfrak{R}[\calC].
	$$
\end{proof}

\begin{remark} If $\calC$ is not object-finite, then we only have a fully faithful functor $\Theta : \Mod\mathfrak{R}\hookrightarrow\rMod\mathfrak{R}[\calC]$. For example, if $X$ is a topological space and $\calP_X$ is the poset of open subsets of $X$, then the classical sheaf theory can be recovered as $\Sh(\calP_X)$, by giving $\calP_X$ a canonical Grothendieck topology. In this case, $\calP_X$ is not object-finite in general.
\end{remark}

\section{Subcategory topologies and module categories} \label{subcat}

In \cite[Expos\'e IV]{SGA}, Grothendieck and Verdier outlined a classification of ``topos finis''. A topos is said to be finite, if it is equivalent to the presheaf topos $\PSh(\calC)$ over a finite category $\calC$. 

\begin{proposition}[\cite{SGA} IV, Exercise 9.1.12] Let $\calC$ be a finite category.
\begin{enumerate} 
\item Every subtopos of a finite topos is again a finite topos.

\item There is an order-reversing bijection between the poset of Grothendieck topologies on $\calC$, and that of the strictly full Karoubian subcategories of $\calC$. 

\item There is an order-reversing bijection between the poset of Grothendieck topologies on $\calC$, and that of the subtopoi of $\PSh(\calC)$.
\end{enumerate}
\end{proposition}

\begin{proof} The proof is a combination of several exercises in \cite[Expos\'e IV]{SGA}, in Sections 7 and 9. It relies on considering the points of topoi, which are geometric morphisms from Set to the presheaf topos in our case. Since it is involved and is far from our purposes, we will reorganize the existing approaches and describe another way to partially demonstrate the facts, which can be extended to ringed finite sites and their module categories.
\end{proof}

The reader can revisit Examples 2.12 and 2.15, where every subcategory is a strictly full Karoubian subcategory, to better understand the statements. 

To establish the above statements, \cite[IV Exercise 9.1.12]{SGA} introduced a topology on $\calC$, for each strictly full subcategory $\calD$. For reference, we recode it below and give the topology a name. Suppose $x\in\Ob\calC$ and $\{x_i\to x\}_i$ is a set of morphisms in $\calC$. For any $w\in\Ob\calC$, there is a natural set map
$$
\coprod_{x_i\to x}\Hom_{\calC}(w,x_i) \to \Hom_{\calC}(w,x).
$$

\begin{definition} Let $\calC$ be a finite category and $\calD$ be a strictly full subcategory. The \emph{subcategory topology} $\calJ^{\calD}$ is given by defining $\calJ^{\calD}(x)$, $\forall x\in\Ob\calC$, to be the following
$$
\{\mbox{sieves}\ S\ \mbox{on}\ x \bigm{|} \coprod_{y\to x \in S}\Hom_{\calC}(w,y) \to \Hom_{\calC}(w,x)\ \mbox{is surjective for all}\ w \in \Ob\calD\}.
$$
\end{definition}

To verify that the above is truly a Grothendieck topology on $\calC$, one can check that the three conditions in Definition 2.1 are satisfied. This is tedious but routine.

\begin{remark} Let $\calC$ be a finite category and $\calD$ be a strictly full subcategory. For each object $x\in\Ob\calC$, the minimal covering sieve (possibly empty) in $\calJ^{\calD}(x)$ is the one generated by all morphisms from objects of $\calD$ to $x$.
\end{remark}

We are interested in the following result, stated in \cite[Expos\'e IV, Exercise 9.1.12 (e)]{SGA}, which makes it possible to classify module categories over ringed finite sites.

\begin{proposition} \cite[Expos\'e IV, Exercise 9.1.12 (e)]{SGA} Let $\calC$ be a finite category.
\begin{enumerate}
\item The subcategory topologies exhausts all possible topologies on a finite category $\calC$.

\item Let $\calD$ be a strictly full subcategory. Under the subcategory topology $\calJ^{\calD}$, there is an equivalence
$$
\Sh(\calC,\calJ^{\calD})\simeq\PSh(\calD).
$$
\end{enumerate}
\end{proposition}

\begin{proof} The first part depends on a discussion on points of topoi, which we shall not dwell into. A relatively elementary proof for Karoubian categories in one direction can be found in \cite[C 2.2 Lemma 2.2.21]{Jo}. As to the second part, by Definition 2.10, the full subcategory $\calD$ equipped with $\calJ^{\calD}|_{\calD}$ is a dense subsite of $(\calC,\calJ^{\calD})$. Moreover, since we must have $1_w\in S$ for any $S\in\calJ^{\calD}(w)$, $\forall w\in\Ob\calD$, $\calJ^{\calD}|_{\calD}$ is the minimal topology on $\calD$. Therefore we have an equivalence $\Sh(\calC,\calJ^{\calD})\simeq\PSh(\calD)$, by the Comparison Lemma.
\end{proof}

The above equivalence can be extended to one between suitable ringed topoi. The later will give us a classification of module categories on ringed finite sites, together with the main result in the preceding section on presheaf topoi.

\begin{theorem} \label{shc} Let $\calC$ be a finite category and $\calD$ be a strictly full subcategory. Let $\mathfrak{R}$ be a sheaf of $k$-algebras on the site $(\calC,\calJ^{\calD})$. Then we have category equivalences
$$
\Mod\mathfrak{R}\simeq\rMod\mathfrak{R}|_{\calD}[\calD].
$$
Here $\mathfrak{R}|_{\calD}$ is the restriction of $\mathfrak{R}$ to $\calD$.
\end{theorem}

\begin{proof} Given the equivalence of topoi
$$
\Psi=(\Psi^{-1},\Psi_*)=(RK_{\iota},Res_{\iota}):\Sh(\calC,\calJ^{\calD}){\buildrel{\simeq}\over{\to}}\PSh(\calD),
$$
and a structure sheaf $\mathfrak{R}$ on $\Sh(\calC,\calJ^{\calD})$, we introduce a morphism of ringed topoi
$$
(\Psi,\Psi^{\sharp}):(\Sh(\calC,\calJ^{\calD}),\mathfrak{R})\to(\PSh(\calD),\mathfrak{R}'),
$$
where $\mathfrak{R}'=\Psi_*(\mathfrak{R})$ and $\Psi^{\sharp}$ is the identity. This is an equivalence of ringed topoi. Hence we obtain an equivalence of module categories 
$$
\Mod\mathfrak{R}\simeq\Mod\mathfrak{R}'.
$$
But $\mathfrak{R}'=\Psi_*(\mathfrak{R})=Res_{\iota}\mathfrak{R}$ is the restriction of $\mathfrak{R}$ along the inclusion functor. Thus $\mathfrak{R}'=\mathfrak{R}|_{\calD}$, and we have desired equivalences
$$
\Mod\mathfrak{R}\simeq\Mod\mathfrak{R}|_{\calD}\simeq\rMod\mathfrak{R}|_{\calD}[\calD],
$$
by Theorem 3.5.
\end{proof}

If, for each $x \in \Ob \calD$, $\mathfrak{R}(x)$ is Noetherian, then $\mathfrak{R}|_{\calD}[\calD]$ is finite-dimensional and the coherent $\mathfrak{R}|_{\calD}[\calD]$-modules are the finitely generated ones \cite{Po}. 

\begin{corollary}\label{cmshc}
Let $\mathbf{C}=(\calC,\mathcal{J})$ be a finite site and $\mathfrak{R}:\calC^{op} \to k\mbox{\rm -Alg}$ be a sheaf of $k$-algebras on $\C$. If for each $x \in \Ob \calD$, $\mathfrak{R}(x)$ is Noetherian, then there is a category equivalence 	
	$$
	\coh \text{-} \mathfrak{R} \simeq \rmod\mathfrak{R}|_{\calD}[\calD],
	$$
for some full subcategory $\calD$ uniquely determined by $\calJ$.
\end{corollary}

\section{Further remarks} 

From our perspective, it is important to understand the role of various finite EI-categories in finite group representations. We end this paper with several remarks, demonstrating promising connections between the present paper and some other works.

\subsection{On dense EI-sites} \label{den}
 
Finite EI-categories frequently occur in group representations and cohomology. We consider sheaves on finite EI-categories under the dense topology, as an example to illustrate methods developed earlier. Moreover, we will see that the sheafification may be used to replace the right Kan extension in Proposition 4.4(2) and Theorem 4.5. The following lemma will help us to compute the sheafifications of presheaves.

\begin{lemma} \label{minsie}  
	Let $\calC$ be a finite {\rm EI}-category and $\calJ_{\rm den}$ the dense topology. Then, for each object $x$, the minimal non-empty covering sieve $S^{\min}_x$ is given by all the morphisms from the minimal objects to $x$, namely
	$$
	S^{\min}_x=\{\textsf{f} \in {\rm Hom}_{\calC}(y,x) ~|~ y \in \Ob (\calC_{\leq x})_{\min}
	\}.
	$$ 
Subsequently $\calJ_{\rm den}=\calJ^{\calC_{\min}}$.
\end{lemma}

\begin{proof}
	Firstly, we verify that $S_x$ is a sieve. For any morphism $\textsf{f}: y \to x \in S_x$ and any morphism $\textsf{g}: z \to y$ with ${\rm cod}(\textsf{g})={\rm dom}(\textsf{f})$, we have to show that $\textsf{fg} \in S_x$. Note that $z \leq y$ and $y$ is minimal, we infer that $z \in \Ob (\calC_{\leq x})_{\min}$. It follows that $\textsf{fg} \in S_x$.

	Secondly, we shall show that $S_x$ is a covering sieve. For any $\textsf{f}:y \to x \in {\rm Hom}_{\calC}(-,x)$, we have to prove that there exists a morphism $\textsf{g}: z \to y$ such that $\textsf{fg} \in S_x$. Since $\calC$ is finite, we can always find such a morphism $\textsf{g}: z \to y$.
	
	Finally, we will show that $S_x$ is minimal in $\calJ_{\rm den}(x)$. That is, we have to check that $S_x \subseteq T$ for any $T \in \calJ_{\rm den}(x)$. For any $\textsf{f}: y \to x \in S_x$, by definition, $y \in \Ob (\calC_{\leq x})_{\min}$. Since $T \in \calJ_{\rm den}(x)$, according to the definition of the dense topology, there exists a morphism $\textsf{g}: z \to y$ such that $\textsf{fg}:z \to x \in T$. Now $y$ is minimal and clearly $[z] \leq [y]$, it follows that $z \cong y$. Thus there exists a morphism $\textsf{h}: y \to z$. It follows immediately that $\textsf{fgh}: y \to x \in T$ as $T$ is a sieve. Evidently, $\textsf{gh} \in {\rm Aut}_{\calC}(y)$, therefore $\textsf{f}=\textsf{fgh}(\textsf{gh})^{-1} \in T$. This proves that $S_x \subseteq T$. The non-emptiness of sieves follows from the definition of the dense topology.
\end{proof}

Equivalently, the unique minimal covering sieve of $\calJ_{\rm den}(x)$ can be characterized as a functor:
\begin{eqnarray}
S_x^{\rm min}(y)=
\begin{cases}
\text{Hom}_{\calC}(y,x),       & \text{if} ~y \in \Ob (\calC_{\leq x})_{\min}; \notag \\
\emptyset, & {\rm otherwise}.
\end{cases}
\end{eqnarray}

\begin{proposition} \label{sheafification}
Let $\mathbf{C}=(\calC,\mathcal{J}_{\rm den})$ be the dense site on a finite {\rm EI}-category. Then, for each presheaf $\mathfrak{F}: \calC^{op} \rightarrow {\rm Set}$ and each $x \in {\rm Ob}\calC$, the half-sheafification of $\mathfrak{F}$ can be obtained as follows (with obvious maps)
$$
\mathfrak{F}^\dag(x) \cong \underset{[y] \subset \Ob (\calC_{\leq x})_{\min}} \prod \underset{i} \biguplus ~~ \mathfrak{F}(y)^{H_i}.
$$ 
where $H_i$'s are the stabilizers of the transitive $\Aut_{\calC}(y)$-sets in the morphism set
$$
{\rm Hom}_{\calC}(y,x)= \underset{i} \biguplus ~~H_i  \backslash
\text{\rm Aut}_{\calC}(y).
$$
Moreover, $\mathfrak{F}^a \cong \mathfrak{F}^\dag$. Consequently, $\mathfrak{F}^\dag$ is a sheaf.

Let $\iota : \calC_{\min}\hookrightarrow\calC$ be the inclusion. Any presheaf $\frakG$ on $\calC_{\min}$ can be naturally extended to a presheaf $\frakG'$ on $\calC$ satisfying $(\frakG')^a\cong RK_{\iota}\frakG$.
\end{proposition}

\begin{proof} We proceed to prove this proposition by direct computation. Given a presheaf $\mathfrak{F}$, let us first compute its half-sheafification. On an object $x \in {\rm Ob}\calC$,
	
\begin{eqnarray*}
		\mathfrak{F}^\dag(x) &=& \lim_{S \in \mathcal{J}(x)}  \text{Nat}(S, \mathfrak{F})    \\
		&\cong& \text{Nat}(S_x^{\rm min}, \mathfrak{F})      \\
		&\cong& \underset{[y] \subset \Ob (\calC_{\leq x})_{\min}} \prod \text{Hom}(S_x^{\rm min}(y), \mathfrak{F}(y))  \\
		& \cong& \underset{[y] \subset \Ob (\calC_{\leq x})_{\min}} \prod \text{Hom}_{\text{Aut}_{\calC}(y)}(\text{Hom}_{\calC}(y,x), \mathfrak{F}(y)) \\		
		&\cong& \underset{[y] \subset \Ob (\calC_{\leq x})_{\min}} \prod \text{Hom}_{\text{Aut}_{\calC}(y)}\left(\underset{i} \biguplus ~~H_i  \backslash
		\text{Aut}_{\calC}(y), \mathfrak{F}(y)\right) \\
		&\cong& \underset{[y] \subset \Ob (\calC_{\leq x})_{\min}} \prod \underset{i} \biguplus ~~\mathfrak{F}(y)^{H_i}.
\end{eqnarray*}
	
Here $[y]$ is the isomorphism class of the object $y$ in $\calC$. The third isomorphism above follows from the characterization of the minimal sieve. Particularly if $y \in \Ob \calC_{\min}$, then $S_y^{\rm min}(y)=\text{Aut}_{\calC}(y)$. Thus 
$$
\mathfrak{F}^\dag(y) \cong \text{Hom}_{\text{Aut}_{\calC}(y)}(\text{Aut}_{\calC}(y), \mathfrak{F}(y)) \cong \mathfrak{F}(y).
$$

For a morphism $\textsf{f} : w \to x$, one can see that $\frakF^\dagger(\textsf{f}):\frakF^\dagger(x)\to\frakF^\dagger(w)$ is given by the restriction on fixed points, because $\textsf{f}$ induces an $\Aut_{\calC}(y)$-map 
$$
\Hom_{\calC}(y,w)\to\Hom_{\calC}(y,x)
$$ 
and subsequently an injective group homomorphism between their stabilizers in $\Aut_{\calC}(y)$.

Repeating the above steps, we get the sheafification of $\mathfrak{F}$:
	
\begin{eqnarray*}
		\mathfrak{F}^a(x)&=&(\mathfrak{F}^\dag)^\dag(x) \\
		&\cong& \text{Nat}(S_x^{\rm min}, \mathfrak{F}^\dag)      \\
		&\cong& \underset{[y] \subset \Ob (\calC_{\leq x})_{\min}} \prod \text{Hom}_{\text{Aut}_{\calC}(y)}(S_x^{\rm min}(y), \mathfrak{F}^\dag(y)) \\
		&\cong& \underset{[y] \subset \Ob (\calC_{\leq x})_{\min}} \prod \text{Hom}_{\text{Aut}_{\calC}(y)}(S_x^{\rm min}(y), \mathfrak{F}(y)) \\
		&\cong& \mathfrak{F}^\dag(x).
\end{eqnarray*}
	
One readily checks that $\mathfrak{F}^\dag$ and $\mathfrak{F}^a$ coincide on morphisms as well. Therefore $\mathfrak{F}^a \cong \mathfrak{F}^\dag$, it follows that $\mathfrak{F}^\dag$ is a sheaf.

In the end, given a presheaf $\frakG$ on $\calC_{\min}$, one can define a presheaf $\frakG'$ on $\calC$ by asking $\frakG'(y)=\frakG(y)$ if $y$ is minimal, and $\frakG'(y)=\emptyset$ otherwise. Now by our previous discussions, both $(\frakG')^a$ and $RK_{\iota}\frakG$ are sheaves on $\C_{\rm den}=(\calC,\calJ_{\rm den})$. By Theorem 4.5, we must have $(\frakG')^a\cong RK_{\iota}\frakG$ because they agree on $\calC_{\min}$.
\end{proof}

\begin{remark} Let $\mathbf{C}=(\calC,\mathcal{J}_{\rm den})$ be the dense site on a finite {\rm EI}-category. 
\begin{enumerate} 
\item We emphasize that the last statement of Proposition 5.2 works for an arbitrary co-ideal $\calD$ of $\calC$. For such a strictly full subcategory, any presheaf on $\calD$ can be extended to a presheaf on $\calC$, as in the above proof. Therefore, the right Kan extension may be computed by the sheafification, which provide another way to compare two sheaf categories.

\item From Proposition 5.2, one can quickly deduce the following equivalence
$$
\Sh(\C_{\rm den}) \simeq \underset{[y] \subset \Ob \calC_{\min}} \prod \rSet{\rm Aut}_{\calC}(y),
$$
which matches with Proposition 4.4(2). Let $\mathfrak{R}$ be a sheaf of $k$-algebras on $\C_{\rm den}$. Passing to the module categories, we obtain an equivalence $\Mod\mathfrak{R} \simeq  \prod_{[y] \subset \Ob \calC_{\min}} \rMod \mathfrak{R}(y)[\Aut_{\calC}(y)]\simeq \rMod \prod_{[y] \subset \Ob \calC_{\min}} \mathfrak{R}(y)[\Aut_{\calC}(y)]$.
\end{enumerate} 
\end{remark}

The preceding result provides some kind of ``block decomposition'' of $\Mod \mathfrak{R}$. 

\begin{example}	Let $G$ be a finite group and $p$ be a prime number such that $p \bigm{|} |G|$. Consider $\mathbf{O}_p^\circ(G)=(\calO_p^\circ(G),\mathcal{J}_{\rm den})$ (see Section 2.4). Let $\underline{k}$ be the constant sheaf given by $k$. Then
$$ 
\Sh(\mathbf{O}_p^\circ(G),k)=\Mod\underline{k} \simeq \prod_P \rMod k(N_G(P)/P),
$$
where $P$ runs over the set of conjugacy classes of order $p$-subgroups.
\end{example}

\subsection{Vertices and sources} In \cite{Xu}, given a finite EI-category $\calC$, one can introduce a vertex (which is a unique full subcategory of $\calC$) and a source for each finite-dimensional indecomposable $k\calC$-module (covariant functor from $\calC$ to $\rMod k$) $\frakM$. Let $\frakN$ be a presheaf on $\calC$. In light of sheaf theory, there is a (unique) finest topology $\calJ_{\frakN}$, for which $\frakN: \calC^{op} \to \rMod k$ becomes a sheaf on $(\calC,\calJ_{\frakN})$. Now we know that $\calJ_{\frakN}=\calJ^{\calD}$ for a unique strictly full subcategory $\calD$. It is interesting to compare $\calD$ with the vertex of $\frakN$ (when it is indecomposable), because they seem to be closely related. Plus, the construction of $\calD$ for an indecomposable module works for any finite category, not just the finite EI-categories. Moreover it is defined for arbitrary modules, even the infinite-dimensional ones.

\subsection{Sipp-topology on $G$-Set} There is a sipp-topology introduced on $G$-Set by Balmer \cite{Ba}. It may be restricted to the finite EI-category $\calO(G)$, the orbit category of $G$. One can verify that the sipp-topology is a subcategory topology on $\calO(G)$, given by the strictly full subcategory $\calO_p(G)$, the $p$-orbit category (see Section 2.4). As a consequence $\Sh(G\mbox{-Set},\calJ_{\rm sipp})\simeq\Sh(\calO(G),\calJ_{\rm sipp})\simeq\PSh(\calO_p(G))$. Since every strictly full subcategory of $\calO(G)$ defines a topology, it would be interesting to learn how it can be used to understand group representations.

\subsection{Hochschild cohomology of $\Mod\mathfrak{R}$} Let $\C$ be a finite site and $\calJ^{\calD}$ be the subcategory topology. Suppose $\mathfrak{R}$ is a sheaf of $k$-algebras on $\C$. Since $\mathfrak{R}|_{\calD}[\calD]$ is an associative algebra, one can consider its Hochschild cohomology. Due to the category equivalence we established earlier, Hochschild cohomology theory for $\mathfrak{R}|_{\calD}[\calD]$ can be considered as one for the Abelian category $\Mod\mathfrak{R}$ on the ringed site $(\C,\mathfrak{R})$, in which $\C=(\calC,\calJ^{\calD})$. In light of Remark 3.6, it is curious to know weather this coincides with the one introduced by Lowen and Van den Bergh \cite{LB} for Abelian categories, where they had module categories over ringed spaces in mind.

\end{document}